\documentclass[10pt]{amsart}

\usepackage[latin1]{inputenc}
\usepackage{amsfonts}
\usepackage{amsmath}
\usepackage{amssymb}
\usepackage{amsthm}
\usepackage[american]{babel}
\usepackage[T1]{fontenc}
\usepackage{xspace}
\usepackage{bbm}
\usepackage[all]{xy}

\theoremstyle{plain} 
\begingroup 
\newtheorem{theorem}{Theorem}[section]
\newtheorem{corollary}[theorem]{Corollary}
\newtheorem{conjecture}[theorem]{Conjecture}
\newtheorem{lemma}[theorem]{Lemma}
\newtheorem{proposition}[theorem]{Proposition}

\endgroup

\theoremstyle{definition}
\newtheorem{definition}[theorem]{Definition}
\newtheorem*{ack}{Acknowledgment}

\theoremstyle{remark}
\newtheorem{remark}[theorem]{Remark}

\newcommand{\Ks}{\ensuremath{K^\times}\xspace}
\newcommand{\MZ}{\ensuremath{\mathbbm{Z}}\xspace}
\newcommand{\MQ}{\ensuremath{\mathbbm{Q}}\xspace}

\newcommand{\MP}{\ensuremath{\mathbbm{P}}\xspace}
\newcommand{\MG}{\ensuremath{\mathbbm{G}}\xspace}
\newcommand{\MF}{\ensuremath{\mathbbm{F}}\xspace}
\newcommand{\MGh}{\ensuremath{\hat{\mathbbm{G}}}\xspace}

\title{On the Grassmannian homology of $\MF_2$ and $\MF_3$}
\author{Oliver~Petras}
\address{Hausdorff Center for Mathematics\\ Universit\"at Bonn\\ Endenicher Allee 62 \\ D--53115 Bonn \\ oliver.petras@hcm.uni-bonn.de}

\author{Dorothee~Richters}
\address{Fachbereich Physik, Mathematik und Informatik\\ Johannes Gutenberg - Universit\"at Mainz\\ Staudinger Weg 9\\ D--55099 Mainz \\ dororich@students.uni-mainz.de}


\begin{document}

\begin{abstract}
We prove the vanishing of the subgroup of Bloch's cubical higher Chow groups $CH^2(\text{Spec}(\MF_p),3)$, $p=2,3$, generated by the images of corresponding projective Grassmannian homology groups ${}^PGH_1^2(\MF_p)$ using computer calculations.
\end{abstract}


\maketitle

\section{Introduction}
As already explored in \cite{HGaSM99, OP09}, it is still an open problem to construct explicit motivic cohomology classes even for the spectrum $\text{Spec}(K)$ of a number field $K$. This note is an attempt to contribute to a slightly easier problem by concentrating on finite fields on the one hand and on motivic cohomology classes described via Bloch's higher Chow groups which come from the ``linearized'' version of motivic cohomology, namely the projective Grassmannian homology of \cite{BPS}, on the other hand.

By brute force computation on a multiprocessor machine, we list all admissible, non-degenerate higher Chow cycles in an acyclic subquotient of Bloch's cubical higher Chow complex $Z^2(\text{Spec}(\MF_p),3)$ for $p=2,3$ which are linearly embedded, i.e. come from projective Grassmannian cycles via the inclusion map of chain complexes of \cite{BPS}. This enables us to prove the main result:
\begin{theorem}
 The subgroup of $CH^2(\text{Spec}(\MF_p),3)$ for $p=2,3$ generated by linearly embedded cycles vanishes. In other words, the projective Grassmannian homology groups ${}^PGH_1^2(\MF_p)$ for $p=2,3$ are trivial.
\end{theorem}
We also obtained partial results on the number of admissible fractional linear cycles in $Z^2(\text{Spec}(\MF_p),3)$ for $p=5,7,11$, but the computation of the whole subgroup in Bloch's higher Chow group in codimension two comprised by fractional linear algebraic cylces was out of reach with our computational ressources.

The complete C++ and SAGE program code is available upon request from any of the authors.

\ack
 We thank Stefan M\"uller-Stach for suggesting this research topic and for supervising the diploma thesis of the second author, on which this article is based. The first author also would like to thank Jens Hornbostel for several discussions.

 Furthermore, we thank the SFB/TR $45$ for financial support and the Hausdorff Center for Mathematics (Universit\"at Bonn) for providing computational facilities without which the results would not be ready yet.

\section{Bloch's higher Chow groups}
Let us begin by recalling the definition of Bloch's higher Chow groups. Since there are many good expositions in literature, we only consider the cubical version keeping in mind that Levine \cite{levine:1994} established a quasi-isomor\-phism to the ``original'' simplicial version due to Bloch \cite{bloch:1986}.

Let $K$ be a field and $$\Box^n_K=(\mathbbm{P}^1_{K}\setminus \{1\})^n$$ with coordinates $(z_1,\ldots,z_n)$ the algebraic standard cube with
 $2^n$ faces of codimension $1$: $$\partial\Box^n_K:=\bigcup_{i=1}^n\{(z_1,\ldots,z_n)\in\Box^n_K|
z_i\in\{0,\infty\}\}$$ and faces of codimension $k$: 
\[
\partial^k\Box^n_K:=\bigcup_{i_1<\ldots <i_k}\{(z_1,\ldots,z_n)\in\Box^n_K|
z_{i_1},\ldots,z_{i_k}\in\{0,\infty\}\}.
\]
In case the field $K$ is not important, we shall drop the subscript in the rest of the article. We now let $X$ be a smooth quasi-projective variety over $K$ and write $Z^p(X,n)=c^p(X,n) / d^p(X,n)$ for the quotient of the free abelian group $c^p(X,n)$ generated by integral closed algebraic subvarieties of codimension $p$ in $X\times \Box^n_K$ which are admissible (i.e. meeting all faces of all codimensions in codimension $p$ again -- or not at all) modulo the subgroup $d^p(X,n)$ of degenerate cycles (i.e. pull-backs of $X\times \text{facets},$ where a facet is a component of $\partial\Box^n$ by coordinate projections $\Box^n\rightarrow\Box^{n-1}$). These groups form a simplicial abelian group:
\renewcommand{\arraystretch}{0.3}
\[
\ldots Z^p(X,3) \begin{array}{l}\rightarrow \\ \rightarrow \\ \rightarrow \\ \rightarrow \end{array} Z^p(X,2) \begin{array}{l} \rightarrow \\ \rightarrow \\ \rightarrow \end{array} Z^p(X,1)\begin{array}{l} \rightarrow \\ \rightarrow \end{array} Z^p(X,0).
\]
\renewcommand{\arraystretch}{1}
\begin{definition}
Bloch's higher Chow groups $CH^p(X,n)$ are the homotopy groups of the above simplicial object or equivalently the homology groups of the above complex with respect to Bloch's boundary map given by 
\[
\partial_B:=\sum_i (-1)^{i-1}(\partial_i^0-\partial_i^\infty),
\]
where $\partial_i^0,\partial_i^\infty$ denote the restriction maps to the faces $z_i=0$ resp. $z_i=\infty$:
\[
CH^p(X,n):=\pi_n(Z^p(X,\bullet))=H_n(Z^p(X,\bullet),\partial_B).
\]
\end{definition}

\begin{theorem}[\cite{Voe}]
Assume that a field $K$ admits resolution of singularities and let $X$ be a smooth quasi-projective variety over $K$. Then Bloch's higher Chow groups are isomorphic to the motivic cohomology groups: 
\[
CH^p(X,n)\cong H_{\mathcal{M}}^{2p-n,p}(X,\MZ). 
\]
\end{theorem}
The higher Chow groups satisfy several formal properties as expected of motivic cohomology. In particular recall the well-known comparison theorem:
\begin{theorem}[\cite{levine:1994},\cite{bloch-ss},\cite{FS}]
Let $X$ be a smooth, quasi-projective variety of dimension $d$ over a field $K$. Let further $gr_\gamma^q K_n(X)$ be the $q$-th piece of the weight filtration of Quillen's $K$--theory of $X$. Then 
\[
gr_\gamma^q K_n(X)\otimes \mathbbm{Z}\left[\frac{1}{(n+d-1)!}\right] \cong CH^q(X,n)\otimes \mathbbm{Z}\left[\frac{1}{(n+d-1)!}\right]. 
\]
More generally, there is a spectral sequence 
\[
CH^{-q}(X,-p-q) \Rightarrow K_{-p-q}(X) 
\]
for an equidimensional scheme $X$ over $K$ abutting to $K$--theory and inducing the above isomorphism after tensoring with $\MQ$.
\end{theorem}
\begin{remark}
In contrast to general algebraic varieties, higher Chow groups are well-known for number fields or finite fields: As remarked in \cite{Wei}, for prime powers $q=p^r$ one has 
\[
CH^2(\MF_q,3):= CH^2(\text{Spec}(\MF_q),3)\cong \MZ/(q^2-1)\MZ.
\]
\end{remark}

\section{Grassmannian homology}
In this section we will quickly review the definitions of the different variants of Grassmannian homology and their main properties from \cite{BPS, Ger}. Let $K$ be some field, and consider the ``coordinate simplex'' in the projective space $\MP^{p+q}_K$ over $K$ given by the $p+q+1$ hyperplanes, which are defined by the vanishing of one of the homogeneous coordinates in $\MP^{p+q}_K$, and their intersections.

\begin{definition}
We call two linear subspaces of $\MP^n_K$ of dimension $r$ and $s$ \textit{transverse}, if their intersection is of dimension less or equal to $r+s-n$. A linear subspace of $\MP^n_K$ of dimension $d$ is said to be transverse if its intersection with any part of the coordinate simplex of $\MP^n_K$ is transverse.
\end{definition}

One knows that the transverse subspaces of $\MP^{p+q}_K$ with given codimension $p$ form a subset $\MGh^p_q$ of the Grassmannian manifold $\MG^p_q$, and that the intersections with the $i$-th coordinate plane defines a map $A_i\colon \MGh^p_q\to\MGh^p_{q-1}, \, i\geq 0$. These maps satisfy $A_i\circ A_{i+1}=0$ for all $i\geq 0$ \cite{BPS}.
\begin{definition}
We denote by ${}^PCG^p_*$ the chain complex
\[
\ldots\xrightarrow{\partial_{p+3}}\MZ\MGh^p_2\xrightarrow{\partial_{p+2}}\MZ\MGh^p_1\xrightarrow{\partial_{p+1}}\MZ\MGh^p_0,
\]
where the differentials are defined by 
\[
\partial_q:=\sum_{i=0}^q(-1)^i\MZ A_i.
\]
The \textit{projective Grassmannian homology groups} of $K$ are given by the homology groups of this complex: 
\[
{}^PGH^p_q:=H_{p+q}({}^PCG^p_*). 
\]
\end{definition}
We are interested in the projective Grassmannian groups because of their connection with Bloch's higher Chow groups. As shown in \cite{BPS}, for any field $K$ there is an inclusion 
\[
{}^PGH_k^p(K)\hookrightarrow CH^p(\text{Spec}(K),p+k),
\]
which we shall now describe in more detail.

For this, we introduce a variant of this definition in the affine case: For all $n\in\mathbbm{N}$ we embed the $n$-dimensional affine space $\mathbbm{A}^n_K$ into $\mathbbm{P}^n_K$, considered as the set of lines in $K^{n+1}$ passing through the origin, by identifying $\mathbbm{A}^n_K$ with the affine hull of the canonical basis $e_1,\ldots, e_{n+1}$ of $K^{n+1}$. Thus, one may think of $\mathbbm{A}^n_K$ as $\mathbbm{P}^n_K\setminus H^n$, where $H^n$ is the hyperplane normal to the vector $(1,\ldots, 1)\in K^{n+1}$. As explained in \cite{Ger}, the intersections of $H^n\subset\mathbbm{A}^{n+1}_K$ with any coordinate hyperplane define $H^{n-1}$ considered as subspace of $\mathbbm{A}^n_K$, so that one obtains a subcomplex of ${}^PCG^p_*$:
\begin{definition}
 The complex $({}^HCG^p_*,\partial)$ is defined as the subcomplex of $({}^PCG^p_*,\partial)$ generated by all linear subspaces of $\MGh^p_*$ which are contained in $H^{p+*}$.

Further, we define ${}^ACG^p_q, q\geq 0,$ as the free abelian group generated by transverse subspaces of codimension $p$ in $\mathbbm{A}^{p+q}_K$, i.e. subspaces in $\mathbbm{P}^{p+q}_K$ of the same codimension whose intersection with $\mathbbm{A}^{p+q}_K$ embedded as described above is not empty. These groups also comprise a complex \cite{Ger}, which we denote by $({}^ACG^p_*,\partial)$. The homology of this complex, ${}^AGH^p_q:=H_{p+q}({}^ACG^p_*)$, is called \textit{affine Grassmannian homology}. 
\end{definition}
\begin{remark}
 As noted in \cite[Rem.3.4]{Ger}, this complex is not a subcomplex of $({}^PCG^p_*,\partial)$ but a quotient.
\end{remark}
\begin{proposition}{\cite[p. 90]{Ger}}
 There is a short exact sequence of chain complexes
\begin{equation}
\label{eq:cplx}
0\to {}^HCG^p_*\to {}^PCG^p_*\to {}^ACG^p_*\to 0. 
\end{equation}
\end{proposition}
\begin{theorem}{\cite[Thm. 3.5]{Ger}}
There is an isomorphism 
\[
{}^AGH^p_q\cong H_{p+q}(\textup{GL}_p(K),\textup{GL}_{p-1}(K)),
\] 
where the pair of general linear groups is given by the inclusion 
\[
\textup{GL}_{p-1}\hookrightarrow \textup{GL}_p,\qquad A\mapsto \begin{pmatrix} A & 0 \\ 0 & 1 \end{pmatrix}.
\]
 In particular, there is an isomorphism for all $i\geq 1$:
\[
{}^AGH^i_0\cong K^M_i,
\]
where the group on the right hand side denotes the $i$-th Milnor $K$-group of $K$.
\end{theorem}
We are interested in a slightly different version of the affine Grassmannian homology (cf. \cite[Sect. 4.2]{Ger}) which is related to Bloch's higher Chow groups introduced in the section before. To define it, we need some more preparation:
\begin{definition}
 A linear subspace of $\mathbbm{A}^n_K$ is said to be \textit{affine transverse} if it does not intersect any $(p-1)$-dimensional stratum of the affine coordinate-simplex. Write $DG^p_q$ for the free abelian group generated by all affine transverse subspaces in $\mathbbm{A}^q_K$ of codimension $p$. One defines a differential $\partial_q\colon DG^p_q\to DG^p_{q-1}$ as (possibly empty) intersection with the faces of $\mathbbm{A}^q_K$.
\end{definition}

\begin{remark}
In other words, this complex is the subcomplex of Bloch's higher Chow complex from \cite{bloch:1986} which computes the higher Chow groups in codimension $p$ given by all chains of linearly embedded cycles. 
\end{remark}
\begin{remark}
It can be shown \cite[Prop. 4.6]{Ger} that ${}^AGH^p_l\cong H_{p+l}(DG^p_*)$.
\end{remark}
So, we can finally describe the map of complexes from the projective Grassmannian complex to Bloch's cycle complex more precisely: In view of the short exact sequence of complexes \eqref{eq:cplx}, a projective Grassmannian homology class is mapped to an affine one and then -- via the isomorphism just cited -- mapped onto a higher Chow cycle with fractional linear coordinate functions.

In general, one expects that Grassmannian homology already computes higher Chow groups of number fields \cite{BPS}:
\begin{conjecture}
 If $K$ satisfies the rank conjecture of Suslin, e.g. if $K$ is a number field, the cubical higher Chow groups $CH^p(K,p+q)$, $q\geq 0$, are generated by fractional linear cycles.
\end{conjecture}
\begin{remark}
The latter conjecture is a theorem of Gerdes for $q=0,1$ by \cite{Ger}.
 \end{remark}

As we are interested in finite fields, these results do not help very much. Here, we show the difference between higher Chow groups and Grassmannian homology groups of finite fields. One already knows (cf. Br\"ahler's thesis \cite{Brae}): 
\begin{proposition}
Let $q$ be a prime power, and assume $d\geq q$. Then 
\[
{}^PGH^d_k(\MF_q)\cong\begin{cases} \MZ^{(q-1)^d-1}\oplus\MZ, & k=0, d\equiv 0\, (2), \\ 0, & \text{else.} \end{cases}
\]
On the other hand
\[
{}^AGH^d_k\MF_q\cong \begin{cases} \MZ^{\frac{(q-1)^{d+1}+(-1)^d}{d}}, & k=0 \\ 0, & \text{else.}  \end{cases}
\]
In addition, Br\"ahler claims that by manual matching he obtained the following result: ${}^PGH^2_1(\MF_3)\cong \MZ^5$ and ${}^PGH^2_0(\MF_3)\cong 0$. 
\end{proposition}
\begin{remark}
 We shall disproof the last assertion of Br\"ahler and show the vanishing of ${}^PGH^2_1(\MF_3)$.
\end{remark}

\section{Explicit cycles in higher Chow groups and Grassmannian homology}
Our aim is to find explicit generators and relations for Grassmannian homology groups. For this we introduce some notation for cycles in higher Chow groups:
\begin{definition}
Given a map $\phi\colon (\mathbbm{P}_K^1)^n\rightarrow (\mathbbm{P}_K^1)^m$ for some $n,m\in\mathbbm{N}$, let $Z_\phi$ be the cycle $\phi_*((\mathbbm{P}_K^1)^n)\cap\Box^m$ associated to $\phi$ in the sense of \cite[sect. 1.4]{Ful}. Then for $x=(x_1,\ldots,x_n)$ one defines
\[
\left[\phi_1(x),\ldots,\phi_m(x) \right]:=Z_{(\phi_1(x),\ldots,\phi_m(x))}.
\]
\end{definition}

\begin{remark}
Unlike the papers \cite{HGaSM99,OP09} we will not only consider the so-called Totaro curves in $Z^2(K,3)$, which have proven sufficient to write down explicit generators for (cubical) higher Chow groups of codimension two over some number fields \cite{OP09}. In this paper we will be concerned with the most general (cubical) fractional linear cycles of the form
\[
\left[\frac{a_1x+b_1y+c_1}{d_1x+e_1y+f_1},\ldots, \frac{a_4x+b_4y+c_4}{d_4x+e_4y+f_4}\right]\in Z^2(K,4)
\]
and
\[
\left[\frac{a_1x+b_1}{c_1x+d_1},\frac{a_2x+b_2}{c_2x+d_2},\frac{a_3x+b_3}{c_3x+d_3} \right]\in Z^2(K,3)
\]
with $a_1,\ldots, f_4\in K$ such that the denominators are all non-zero.
\end{remark}
The idea of the present article is to consider the image of the Grassmannian homology of finite fields inside all admissible cubical higher Chow cycles, and in particular its cokernel. An analogous map has been studied for number fields, and in the simplicial setting in \cite{Ger}. There it is also shown that this analogous map induces a rational isomorphism for fields satisfying the rank conjecture.

\begin{remark}
\label{rem:admissible}
In the rest of the article one especially has to care about the admissibility condition mentioned in the survey on higher Chow groups. Totaro cycles are known to be admissible. Moreover, a cycle $\mathcal{Z}=[f(x),g(x),h(x)]$ is admissible if and only if every zero or pole of one of the rational functions which is also a zero or pole of another one is contained in the preimage of $1$ of the third function.

Later in this paper we will discuss admissibility issues. In particular, the algorithmic check for admissibility will be explained.
\end{remark}

To simplify our computations in the quotient $Z^2(K,3) / \partial Z^2(K,4)$, we divide out an acyclic subcomplex of $Z^2(K,\bullet)$ consisting of cycles with a constant coordinate on the left-hand side.
\begin{lemma}
 The following subcomplex of $Z^2(K,\bullet)$ is acyclic:
\begin{align*}
Z'(K,\bullet):= &\ldots\rightarrow Z^1(K,1)\otimes Z^1(K,3)\rightarrow Z^1(K,1)\otimes Z^1(K,2) \\
&\rightarrow Z^1(K,1)\otimes \partial Z^1(K,2) \rightarrow 0
\end{align*}
\end{lemma}
\begin{proof}
Based on \cite[p. 326-327]{nart:1995}.
\end{proof}
\begin{definition}
 We set $C^2(K,\bullet):= Z^2(K,\bullet) / Z'(K,\bullet)$.
\end{definition}
\begin{remark}
In order to further simplify and speed up computations, we choose to compute a variant of projective Grassmannian homology: A generic fractional linear cycle in $C^2(K,4)$ looks like 
\[
\left[\frac{a_1x+b_1xy+c_1y+d_1}{e_1x+f_1xy+g_1y+h_1},\ldots,\frac{a_4x+b_4xy+c_4y+d_4}{e_4x+f_4xy+g_4y+h_4}  \right], \, a_1,b_1,\ldots, h_4\in K.
\]
We shall restrict ourselves to the case that the coefficients of $xy$ all vanish. Certainly this enlarges the quotient $C^2(K,3) / \partial C^2(K,4)$, i.e. there are the same curves in $C^2(K,3)$ but less relations among them. But as we will show the vanishing of the homology of the enlarged quotient, the vanishing of the smaller quotient is guaranteed.
\end{remark}

\section{Some lemmas}
\label{sec:lemmas}
In this section, we briefly recall some results from \cite[Sect. 4]{OP09} which we will make use of in the sequel:
\begin{proposition}{\cite[Prop. 4.6]{OP09}}
 Let $f, g, h_1, h_2$ be rational functions of one variable $x$ such that all cycles occurring are admissible. Then the following identities hold in $C^2(K,3)/\partial C^2(K,4)$:
\[
\left[f(x),g(x),h_1(x)h_2(x)\right] = \left[f(x),g(x),h_1(x)\right]+\left[f(x),g(x),h_2(x)\right].
\]
\end{proposition}
This immediately implies:
\begin{corollary}
 Let $f,g$ be rational functions of one variable $x$, further let $\zeta_n\in \Ks$ be some $n$-th root of unity in \Ks such that all cycles occurring are admissible: Then the following identities hold in $C^2(K,3)/\partial C^2(K,4)$:
\begin{align*}
n \left[x,f(x),\zeta_n\right] &= 0,  \\
n\left[ x, f(x), \zeta_ng(x)\right] &= \left[ x, f(x), (g(x))^n\right], \\
\left[x,f(x),g(x) \right] &= - \left[x,f(x),\frac{1}{g(x)} \right]. 
\end{align*}
\end{corollary}
Furthermore, we have

\begin{proposition}{\cite[Prop. 4.13]{OP09}}
 Let $f, g, h$ be rational functions in one variable, and let all of the cycles be admissible. Then the following identity holds in the quotient $C^2(K,3)/\partial C^2(K,4)$:
\[
\left[f(x),g(x),h(x)\right] = -\left[f(x),h(x),g(x)\right]+\hspace{-0.5 em}\sum_{x_0\in div(f)}\hspace{-0.8 em}\pm Z(h(x_0),g(x_0)).
\]
Note that the sign is positive if $x_0$ is a zero and negative if $x_0$ is a pole of $f$.
\end{proposition}

\section{Computational issues}
Now we explain the algorithm used to compute Grassmannian homology of finite fields. We used C++ as programming language and the GNU C++ compiler g++, version 4.3.2 coming with the open source operating system OpenSUSE 11.1. The linear algebra computations were done with the help of the open source computer algebra system SAGE \cite{sage}.

The computation of the Grassmannian homology groups in general can be split into several parts:
\begin{enumerate}
 \item Collect all admissible linearly embedded cycles in $C^2(K,3)$.
\item Compute a basis for the so-called Grassmannian cycles which are given by the kernel of $\partial\colon C^2(K,3)\to C^2(K,2)$ restricted to these fractional linear algebraic cycles in codimension two.
\item Compute a basis for the quotient $C^2(K,3)/\partial C^2(K,4)$ consisting of fractional linear cycles. This amounts to collecting admissible fractional linear cycles in $C^2(K,4)$, computing their boundaries, and computing the structure of the quotient using a Smith normal form.
\item Intersect the basis of the Grassmannian cycles with the basis of the quotient.
\end{enumerate}
In the following subsections we will describe the different steps in more detail, and in the next section we will present our results for finite fields $K=\MF_p$, $p= 2,3$.
 
\subsection{Computing admissible algebraic cycles in $C^2(K,3)$}
Using nested loops we checked all fractional linear cycles of the form 
\[
\left[\frac{a_1x+b_1}{c_1x+d_1},\frac{a_2x+b_2}{c_2x+d_2},\frac{a_3x+b_3}{c_3x+d_3} \right]\in C^2(K,3)
\]
for parameters $a_1,\ldots, d_3\in K$ for the admissibility condition mentioned in remark \ref{rem:admissible}.

To decrease the number of nested loops, we reparametrized cycles of the above shape by some M\"obius transformation such that the first coordinate simply reads $x$. Thus we are left with cycles of the form $\left[x, \frac{a_1x+b_1}{c_1x+d_1},\frac{a_2x+b_2}{c_2x+d_2} \right]\in C^2(K,3)$. Further, we decrease the number of iterations by ``normalizing'' coordinate functions: fractions of the form $\frac{ax+b}{cx+d}$ by are scaled by $d$, if $d\neq 0$, obtaining $\frac{ad^{-1}x+bd^{-1}}{cd^{-1}x+1}$. If $d = 0$, we scale by $c$ to obtain $\frac{ac^{-1}x+bc^{-1}}{x}$. The case $c = d = 0$ cannot occur because such a cycle is not admissible.

The check for the admissibility, non-triviality and (non-)degeneracy of these cycles is divided into several steps in the program: Check that
\begin{itemize}
 \item no coordinate equals $1$ (otherwise this cycle would not be contained in $\Box^3$),
\item (before the reparametrization) the leftmost coordinate is not constant (otherwise the cycle would be contained in the acyclic subgroup $Z^1(K,1)\otimes Z^1(K,2)\subset Z^2(K,3)$, which we divided out),
\item no coordinate equals $0$ or $\infty$ (otherwise the codimension of the intersection with $\Box^3$ would be wrong),
\item at most one coordinate is constant (otherwise the cycle would be degenerate),
\item if some $a\in K$ occurs more than once among the zeroes or poles of the coordinates, this $a$ is contained in the preimage of $1$ under the third coordinate function (otherwise the intersection of the cycle with $\Box^3$ has the wrong codimension).
\end{itemize}

\subsection{Computing Grassmannian cycles}
Having collected all admissible cycles in $C^2(K,3)$, we compute their boundaries. A generic cycle in $C^2(K,3)$ has a boundary of the form $\sum_i n_i(a_i,b_i)$ with $n_i\in \Ks, (a_i,b_i)\in C^2(K,2)$. Since we are working over finite fields, the points in $C^2(K,2)$ can be enumerated, and we assign to them the index $a_ib_i\in\Ks$. Note that this already implies the relation $(a_i,b_i)=(b_i,a_i)\in C^2(K,2)$.

In this way, we assign to each admissible cycle in $C^2(K,3)$ a vector of length $(\# \Ks-1)^2$ with nonzero entries $n_i$ at the indices $a_i b_i$. Note that we subtract $1$ for boundary points with a $1$ in one coordinate.

Assembling these row vectors in a matrix gives a matrix presentation of the Grassmannian boundary operator, whose kernel can be computed with SAGE.

More precisely, the image of this boundary operator lies in the quotient $C^2(K,2)=Z^2(K,2) / (Z^1(K,1)\otimes \partial Z^1(K,2))$, in which all linear combinations of points of the form $\sum_in_i\left([a,b]+[a,c]-[a,bc]\right)$ for $a,b,c\in K^\times$ vanish. Therefore, we have to invoke SAGE to compute the kernel of the morphism between finitely generated modules over $\MZ$.

\subsection{Finding Grassmannian boundaries}
As before, we use nested for-loops to build all possible elements in $C^2(K,4)$ of the form 
\[
\left[\frac{x}{y},\frac{a_1x+b_1y+c_1}{d_1x+e_1y+f_1},\frac{a_2x+b_2y+c_2}{d_2x+e_2y+f_2},\frac{a_3x+b_3y+c_3}{d_3x+e_3y+f_3}\right].
\]
Note that we already reparametrized as before in a way that the first coordinate is equal to $\frac{x}{y}$. Note also that one can again normalize the rational functions in a way that $a_i,\ldots, e_i, i=1,2,3$, runs through $0,\ldots, p-1$ whereas $f_1,f_2,f_3$ only take values in $\{0,1\}$.

The computation of the boundaries of all these possible fractional linear cycles is again split up into parts, i.e. subroutines in the program:
\begin{itemize}
\item Define a new variable $z:=\frac{x}{y}$ and substitute this in all coordinate functions: if then three or more coordinates depend on $z$ only, the cycle is degenerate and thus useless.
\item Check for right codimension of the intersection with $\Box^2$: The intersection with a face of $\Box^4$ of codimension $2$ must either be a point, i.e. must not depend on any of the variables any more, or one of the coordinates must be equal to $1$. This check is done by computing either the poles or the zeroes of two distinct coordinates and substituting these values in the remaining two coordinates.
\item To assure admissibility, $\infty$ may occur at most three times as zero or pole in the same variable. If this is not the case, there will appear inadmissible boundary terms which have $0$ or $\infty$ as zero and pole, but no coordinate equal to $1$ to guarantee admissibility.
\item Then compute the zeroes and poles of all coordinates of the fractional cycle in $C^2(K,4)$ and store them in a list. Also check if three or more coordinates are independent of the same variable: In this case the cycle would be degenerate.
\item Check another case of degeneracy: determine the most frequent expression of the form $ax+by$ with $a,b\in K^\times$ in the coordinates of the cycle in $C^2(K,4)$. Then plug in $z:=ax+by$ into the remaining coordinates and check whether the algebraic cycle then is degenerate, i.e. at least three of the four coordinates depend on $z$ only. If yes, it is again useless.
\item Then compute the boundary of the algebraic cycle in $C^2(K,4)$ as a list of triples of rational functions in one variable. Check the boundary terms for degeneracy as in the case of cycles in $C^2(K,3)$: In case one of the boundary terms is degenerate, the original cycle is so, as well. Also store the corresponding signs of the boundary terms in a second list.
\item Having passed all tests, one has to scale the cycles as explained before. Then one needs to check if two cycles are equal with same or inverse signs and if so: delete one copy or both and modify where necessary the corresponding list of signs. After that the computation of the boundary is finished: In particular, it is checked whether the boundary terms in $C^2(K,3)$ are non-degenerate, nonzero and do not cancel each other.
\end{itemize}

\begin{remark}
 The correct codimension of the intersection of a fractional linear cycle with a face of $\Box^4$ of codimension $3$ need not be checked since the correct intersection with all faces of codimension $2$ already implies this.
\end{remark}

The admissible, non-degenerate cycles in $C^2(K,4)$ and their boundaries are now determined. We proceed by storing the coordinates of the boundary terms in matrices: 
\[
\frac{ax+b}{cx+d}\mapsto \begin{pmatrix} a & b \\ c & d\end{pmatrix},
\]
which are normalized such that the matrix corresponding to the first coordinate is equal to the identity matrix $ \begin{pmatrix} 1 & 0 \\ 0 & 1\end{pmatrix}$. According to the results in section \ref{sec:lemmas} the normalized boundary terms are simplified, united and finally stored in a list.

\subsection{Computing homology}
The main step consists of constructing a relation matrix: Each normalized fractional cycle occurring among the boundary terms of the admissible, non-degenerate cycles in $C^2(K,3)$ is identified with one column in a big matrix, and each of these cycles in $C^2(K,4)$ is identified with one row. The entries in this matrix are given by the coefficients of the terms in $C^2(K,3)$ occurring in the boundary of one of the terms in $C^2(K,4)$.

Computing the Smith normal form, and in particular the elementary divisors, of this huge matrix with SAGE determines the group structure of the quotient $C^2(K,3) / \partial C^2(K,4)$, as one knows from \cite[sec. 2.4.4]{Coh}.
\begin{remark}
In order to keep the dimension of this huge matrix as small as possible, we have to check if several columns represent the same fractional linear cycle in $C^2(K,3)/\partial C^2(K,4)$. All these columns are added to the column with the lowest index, the others are deleted. It is also checked that there are no double rows. If there are any multiple rows, only the one with lowest index is kept, while the others are deleted. Note that we do not check for linear dependence but only for multiple occurrences, since computing the Smith form with every new row in this matrix takes far too long.
\end{remark}

\section{Results}
\subsection{$\MF_2$:}
We find $8$ admissible, non-degenerate algebraic cycles in $C^2(\MF_2,3)$:
\begin{alignat*}{4}
&\left[x,\frac{1}{x+1},\frac{x}{x+1} \right], 
&\left[x,\frac{1}{x+1},\frac{x+1}{x} \right], 
&\left[x,\frac{x}{x+1},\frac{1}{x+1} \right], 
&\left[x,\frac{x}{x+1},x+1 \right], \\
&\left[x,x+1,\frac{x}{x+1} \right], 
&\left[x,x+1,1+\frac{1}{x} \right],
&\left[x,1+\frac{1}{x},\frac{1}{x+1} \right], 
&\left[x,1+\frac{1}{x},x+1 \right].
\end{alignat*}
As one can already check by hand, each of these fractional cycles lies in the kernel of the Grassmannian boundary operator.
\begin{remark}
Invoking the two propositions from section \ref{sec:lemmas}, one easily shows that all $8$ cycles can be transformed into the shape $\left[x,x+1,1+\frac{1}{x} \right]$, which one recognizes as the Totaro curve from \cite{totaro:1992}, also playing an important role in \cite{HGaSM99} and \cite{OP09}.
\end{remark}
Further, we obtain a $163\times 8$ - matrix of relations between the admissible fractional cycles. Let us look at the submatrix consisting of the rows $5,24,25,33,34,37,43$ and $51$:
\[
\begin{pmatrix}
0 & 0 & 0 & 0 & -1 & 0 & 0 & 0\\
0 & 0 & 0 & 1 & 0 & 0 & 0 & 0\\
-1 & 0 & 0 & 0 & 0 & 0 & 0& 0\\
0 & 0 & 0 & 0 & 0 & 0 & 1 & 0\\
0 & 0 & 0 & 0 & 0 & 0 & 0 & 1\\
0 & 0 & 0 & 0 & 0 & -1 & 0 & 0\\
0 & 0 & 1 & 0 & 0 & 0 & 0 & 0\\
0 & 1 & 0 & 0 & 0 & 0 & 0 & 0
\end{pmatrix}
\]
Thus, all $8$ fractional linear cycles in $C^2(\MF_2,3)$ are boundaries of admissible fractional cycles in $C^2(\MF_2,4)$. As our simplified Grassmannian homology contains the projective Grassmannian homology, we have shown:
\begin{proposition}
 The image of ${}^PGH_1^2(\MF_2)$ in $CH^2(\MF_2,3)$ vanishes.
\end{proposition}
\begin{remark}
 As computed in \cite{Brae}, the group ${}^PGH_1^2(\MF_2)$ itself already vanishes.
\end{remark}

\subsection{$\MF_3$:}
Already this case is far more memory consuming. Here, we obtain $64$ admissible, non-degenerate algebraic cycles in $C^2(\MF_3,3)$. Further, we find $106.845$ admissible cycles in $C^2(\MF_3, 4)$ yielding a number of $13.481$ different relations in the quotient $C^2(\MF_3,3) / \partial C^2(\MF_3,4)$. Computing the elementary divisors of the resulting relation matrix gives a sequence
\[
\underbrace{\left[1,\ldots,1\right]}_{\text{$64$ times}},
\]
from which we may conclude that $C^2(\MF_3,3) / \partial C^2(\MF_3,4)$ is trivial. Therefore, we do not need to care about the number of Grassmannian cycles in $C^2(\mathbb{F}_3,3)$ and obtain:
\begin{proposition}
 The image of ${}^PGH_1^2(\MF_3)$ in $CH^2(\MF_3,3)$ vanishes.
\end{proposition}
\begin{remark}
 This result contradicts the claim of \cite{Brae} that the Grassmannian homology group should be isomorphic to $\MZ^5$. But as we can write down explicit cycles in $C^2(\mathbb{F}_3,4)$ which bound to the Grassmannian cycles in $C^2(\mathbb{F}_3,3)$, we are convinced about the correctness of our result.
\end{remark}

\subsection{$\MF_p, \, p=5,7,11$:}
For $\MF_5$ we obtain $2.120$ admissible, non-degenerate, fractional linear algebraic cycles in $C^2(\MF_5,3)$. Already at this point the computation of all those admissible, non-degenerate algebraic cycles in degree $4$ of this complex would have taken far too long to finish: Given the fact that the computations for $\MF_3$ took several weeks and that the number of possible coordinates for cycles in $C^2(\MF_5,4)$ that have to be checked increases by a factor of around $2.000$, we decided to stop our investigations at this point.

Just for the records, we found $18.260$ admissible, non-degenerate, fractional linear algebraic cycles in $C^2(\MF_7,3)$ and a total of $530.496$ admissible, non-degenerate, fractional linear algebraic cycles in $C^2(\MF_{11},3)$.


\end{document}